\documentclass[11pt, a4paper]{amsart}

      \usepackage{amssymb}
\usepackage{color}

      \theoremstyle{plain}
      \newtheorem{theorem}{Theorem}[section]
      \newtheorem{lemma}[theorem]{Lemma}
      \newtheorem{corollary}[theorem]{Corollary}

      \theoremstyle{definition}
      
      \newtheorem*{remark}{Remark}

\numberwithin{equation}{section}
\def\co{\colon\thinspace}

      \makeatletter
      \def\@setcopyright{}
      \def\serieslogo@{}
      \makeatother

\usepackage[arrow, matrix, curve]{xy}

\begin{document}

%


   \author{Sungwoon Kim}
   \address{School of Mathematics,
   KIAS, Hoegi-ro 85, Dongdaemun-gu,
   Seoul, 130-722, Republic of Korea}
   \email{sungwoon@kias.re.kr}

   \author{Thilo Kuessner}
   \address{School of Mathematics,
   KIAS, Hoegi-ro 85, Dongdaemun-gu,
   Seoul, 130-722, Republic of Korea}
   \email{kuessner@kias.re.kr}





   \title[Compact manifolds with amenable boundary]{Simplicial volume of compact manifolds with amenable boundary}


\begin{abstract}
Let $M$ be the interior of a connected, oriented, compact manifold $V$ of dimension at least $2$. If each path component of $\partial V$ has amenable fundamental group, then we prove that the simplicial volume of $M$ is equal to the relative simplicial volume of $V$ and also to the geometric (Lipschitz) simplicial volume of any Riemannian metric on $M$ whenever the latter is finite. As an application we establish the proportionality principle for the simplicial volume of complete, pinched negatively curved manifolds of finite volume.
\end{abstract}




   \thanks{}
   \thanks{}

   \dedicatory{}

   \date{}


   \maketitle



    \section{Introduction}

The simplicial volume $\| M \|$ of a manifold $M$ was introduced by Gromov \cite{Gro82}, it estimates how efficiently the fundamental class of $M$ can be represented by singular simplices.
Even though its definition relies only upon the topological structure of $M$, it surprisingly contains much information about the geometric nature of $M$ and has played an important role in the study of the minimal volume, rigidity and so on.

There are different definitions of the simplicial volume for closed manifolds, open manifolds and compact manifolds with boundary. In each case it will be defined as the infimum of an $l^1$-norm $\parallel \sum_{i\in I} a_i\sigma_i\parallel_1=\sum_{i\in I}\mid a_i\mid$, where the infimum is taken over all chains representing a suitably defined fundamental class in some homology theory with real coefficients. 

To give the precise definitions, let $M$ be an oriented, connected, $n$--dimensional manifold. Throughout this paper we will always use homology groups with $\mathbb R$-coefficients. If $M$ is a {\em closed manifold}, then its {\em simplicial volume} $\| M \|$ is defined as the $\ell^1$--seminorm of
the fundamental class in the singular homology $H_n(M)$ of $M$ with $\mathbb{R}$--coefficients.
If $M$ is a {\em compact manifold with boundary}, then its
{\em relative simplicial volume} $\| M, \partial M \|$ is defined as the $\ell^1$--seminorm of the relative fundamental class in the relative singular homology $H_n(M,\partial M)$ with $\mathbb{R}$--coefficients.
If $M$ is an {\em open manifold}, then its {\em simplicial volume} $\| M \|$ is
defined as the $\ell^1$--seminorm of the locally finite fundamental class in the locally finite homology $H^\mathrm{lf}_n(M)$ of $M$. Moreover, Gromov introduced the notion of the geometric simplicial volume of Riemannian manifolds. The {\em geometric simplicial volume} $\| M,g_M \|_\mathrm{Lip}$ of a Riemannian manifold $(M,g_M)$ is defined as the infimum over the $\ell^1$--norms of locally finite fundamental cycles with finite Lipschitz constant. If $M$ is a closed Riemannian manifold, then one actually has $\| M,g_M\|_\mathrm{Lip}=\| M\|$, so the geometric simplicial volume of a closed manifold does not depend on the Riemannian metric. On the other hand, the geometric simplicial volume of open manifolds may depend on their Riemannian metric.

Gromov established a remarkable {\em proportionality
principle} relating the simplicial volume and the volume of closed Riemannian manifolds.

\begin{theorem}[Gromov \cite{Gro82}]\label{thm:1.1}
If $(M,g_M)$ and $(N,g_N)$ are closed Riemannain manifolds with isometric universal covers, then
$$\frac{\| M \|}{\mathrm{Vol}(M,g_M)}= \frac{\| N \|}{\mathrm{Vol}(N,g_N)}.$$
\end{theorem}

The proportionality principle fails in general for open manifolds. However, the proportionality principle holds for the geometric simplicial volume under some condition.
In fact, Gromov \cite{Gro82} showed that the proportionality constants for a closed Riemannian manifold and a complete Riemannian manifold of finite volume whose universal covers are isometric are
the same. We refer the reader to \cite{LS09-2} for a detailed proof.
To prove the proportionality principle for the simplicial volume it is thus desirable to have an equality between simplicial volume and geometric simplicial volume.

Let us consider open manifolds $M$ which are interiors of compact manifolds $V$ with boundary.
Then $\| V,\partial V\|\le \| M\|$ by \cite[Section 1.1]{Gro82},
but in general one does not have equality. (In fact there are many examples with even $\| M\|=\infty$, see \cite[Theorem 6.1]{Lo07}.) Thus one needs additional assumptions to prove $\| V,\partial V\|=\| M\|$.

The vanishing of the simplicial volume of a manifold is often implied by amenability conditions on the manifold.
For instance, the simplicial volume of any manifold with amenable fundamental group is zero.
This follows from the fact that the bounded cohomology of any amenable group vanishes \cite[Section 3.1]{Gro82}. More generally, Gromov \cite[Section 4.2]{Gro82} proved that if there exists an amenable open covering of $M$ such that every point of $M$ is contained in no more than $n$ subsets for $n=\dim M$ and the covering is amenable at infinity, then the simplicial volume of $M$ vanishes.

Such vanishing results suggest that amenable subsets can be ignored from the point of view of the simplicial volume. In this paper we make this precise
for the simplicial volume of an open manifold that is homeomorphic to the interior of a compact manifold with amenable boundary.
Let $V$ be a compact manifold with amenable boundary and $M$ be its interior, equipped with a given Riemannian metric $g_M$. Under some more geometric and topological conditions\footnotemark\footnotetext[1]{The additional assumption needed in \cite{LS09-1} is the following: \\
{\em Each connected component $W_i$
of $\partial V$ has a finite cover $\overline{W}_i\rightarrow W_i$ that has a self-map $f_i \co \overline{W}_i\rightarrow\overline{W}_i$ with $deg\ f_i\not\in\left\{.1, 0, 1\right\}$ with the following property: Let $f_i^{(k)}$ denote the $k$--fold composition of $f_i$. There is $C > 0$
such that for every $k\ge 1$ the map
$$F_i^{(k)} \co \overline{W}_i \times\left[0, 1\right] \rightarrow\overline{W}_i\times\left[k, k + 1\right] , \left(x, t\right)\rightarrow \left(f_i^{(k)}(x) , t + k\right)$$
has Lipschitz constant at most $C$ with respect to the metric on $\overline{W}_i\times\left(0,\infty\right)$ induced by the metric $d$ on $M$. }} on $V$, L\"{o}h and Sauer \cite[Theorem 1.5]{LS09-1} show that
\begin{eqnarray}\label{eqn:1.1}
\| V ,\partial  V \| = \| M \| = \|M,g_M \|_\mathrm{Lip}.
\end{eqnarray}
The result of \cite{LS09-1} includes the cases of $\mathbb{R}$--rank one locally symmetric spaces and Hilbert modular varieties. However it is natural to conjecture that
Equation (\ref{eqn:1.1}) should hold without the additional assumptions on $V$ and this is
the main objective of this paper: we verify that the amenability condition on the boundary of $V$ is enough to prove (\ref{eqn:1.1}).

Our proof builds on a result about bounded cohomology.
Let $i \co (V,\emptyset) \rightarrow (V, \partial V)$ be the canonical inclusion map. We now suppose that each path component of $\partial V$ has amenable fundamental group. Gromov's relative mapping theorem \cite[Section 4.1]{Gro82} does not imply that the pullback map $i^* \co H^*_b(V,\partial V) \rightarrow H^*_b(V)$ is an isometric isomorphism of bounded cohomology groups. Nonetheless, we show that this map is an isometric isomorphism. Our argument uses Gromov's theory of multicomplexes from \cite[Section 3.2]{Gro82}.

\begin{theorem}\label{thm:1.2}
Let $(X,A)$ be a pair of topological spaces. Assume that each path
component of $A$ has amenable fundamental group.
Then inclusion $\left(X,\emptyset\right)\rightarrow \left(X,A\right)$ induces an isometric isomorphism
$$H_b^*(X,A)\rightarrow H_b^*(X)$$
in degrees $*\ge 2$.
\end{theorem}

The simplicial volume of open manifolds can be described in terms of bounded cohomology. Indeed, the $\ell^1$--seminorm on $\ell^1$--homology can be expressed by the $\ell^\infty$--seminorm on bounded cohomology, and any isometric isomorphism in bounded cohomology then induces an isometric isomorphism in $\ell^1$--homology too. For more details, see \cite{Lo07}.
Such duality principle enables us to verify the following theorem.

\begin{theorem}\label{thm:1.3}
Let $V$ be a connected, oriented, compact $n$--manifold and let $M$ be the interior of $V$.
 If the inclusion $(V,\emptyset) \rightarrow (V,\partial V)$ induces an isometric isomorphism $H^n_b(V,\partial V) \rightarrow H^n_b(V)$ and $\| M\|$ is finite, then
$$\| V,\partial V \| = \| M\|.$$
If moreover, $g_M$ is a Riemannian metric on $M$ and $\| M, g_M\|_\mathrm{Lip}$ is finite, then
$$\| V,\partial V \| = \| M\| = \| M,g_M \|_\mathrm{Lip}.$$
\end{theorem}

Combining Theorem \ref{thm:1.3} with Theorem \ref{thm:1.2}, we have
\begin{corollary}\label{cor:1.4}
Let $V$ be a connected, oriented, compact manifold of dimension at least $2$ and let $M$ be the interior of $V$. If each path component of the boundary of $V$ has amenable fundamental group, then
$$\| V,\partial V \| = \| M\|.$$
If moreover, $g_M$ is a Riemannian metric on $M$ and $\| M, g_M\|_\mathrm{Lip}$ is finite, then
$$\| V,\partial V \| = \| M\| = \| M,g_M \|_\mathrm{Lip}.$$
\end{corollary}

A complete pinched negatively curved manifold $M$ of finite volume is tame, i.e., homeomorphic to the interior of a compact manifold $V$. Furthermore, it is well known that every boundary component of $V$ has a virtually nilpotent fundamental group which injects into $\pi_1V$. Hence, we have the following corollary immediately.
\begin{corollary}\label{cor:1.5}
Let $(M,g_M)$ be a connected, oriented, complete, pinched negatively curved manifold of finite volume that is homeomorphic to the interior of a compact manifold $V$. Then, $$\| V, \partial V \| = \| M \| = \| M,g_M\|_\mathrm{Lip}.$$
\end{corollary}

Now, we establish the proportionality principle for the ordinary simplicial volume of complete pinched negatively curved manifolds of finite volume.
\begin{corollary}\label{cor:1.6}
Let $(M,g_M)$ and $(N,g_N)$ be connected, complete, pinched negatively curved manifolds of finite volume whose universal covers are isometric. Then
$$\frac{\| M \|}{\mathrm{Vol}(M,g_M)} = \frac{\| N \|}{\mathrm{Vol}(N,g_N)}.$$
\end{corollary}

L\"{o}h and Sauer \cite{LS09-2} established the proportionality principle
for the geometric simplicial volume of complete, nonpositively curved Riemannian manifolds. Corollary \ref{cor:1.6} follows from this
proportionality principle and Corollary \ref{cor:1.5}.

\section{Bounded cohomology}

In this section we will prove Theorem \ref{thm:1.2}. (An independent proof of this result for pairs of countable CW complexes\footnotemark\footnotetext[2]{\cite{BBIPP} builds on \cite{Iva87} which states its results for countable CW complexes only. The proofs in \cite{Iva87} can be applied verbatim to spaces possessing a universal covering space (i.e., spaces which are locally path connected and semi-locally simply connected) but beyond that the methods do not seem to apply. Thus the arguments of \cite{BBIPP} should extend to pairs of topological spaces $(X,Y)$ with $X$ possessing a universal cover.} is also given by Bucher-Burger-Frigerio-Iozzi-Pagliantini-Pozzetti in \cite[Theorem 1]{BBIPP}. That paper also gives some further applications e.g.\ to generalized Dehn fillings or for a direct proof of Gromov's equivalence theorem.)
Our proof works for arbitrary topological spaces and it uses Gromov's theory of multicomplexes.

Though we think that for reading the proof in Section \ref{sec:2.3} it should be sufficient to use the results about multicomplexes as black boxes, for the convenience of the reader we will recall all the
relevant definitions in Section \ref{sec:2.1}. All of Section \ref{sec:2.1} is taken almost literally from \cite{Ku} and is originally due to \cite{Gro82}.

In Section \ref{sec:2.2}. we will prove a lemma which in this generality is necessary only for handling the case that $\pi_1A$ does not inject into $\pi_1X$.

A reader who is only interested in the proof of Theorem \ref{thm:1.2} for $\pi_1$--injective subsets (which would actually be sufficient for the applications in Corollaries \ref{cor:1.4} and \ref{cor:1.5}) might skip Section \ref{sec:2.2} and go directly to the proof of Theorem \ref{thm:1.2} in Section \ref{sec:2.3}.

\subsection{Preliminaries about Multicomplexes}\label{sec:2.1}
Multicomplexes are a tool used in Section $3$ of \cite{Gro82}
to investigate the bounded cohomology of topological spaces.
(An alternative route to bounded cohomology is Ivanov's group-cohomological approach in \cite{Iva87}, which however seems to work only for the slightly restricted class of locally path-connected, semi-locally simply connected topological spaces.) Multicomplexes satisfy all axioms of
simplicial complexes except that (other than in simplicial complexes) in a multicomplex there may be more than one $n$--simplex with a given $n-1$--skeleton. More precisely, let $V$ be a set, then a {\em multicomplex} $K=(V,S,\left\{\partial_i\right\}_{i\in\mathbb N})$ with vertices $V$ is given by the following data:
\begin{itemize}
\item[-] a (possibly empty) set $S_{v_0,\ldots,v_n}$ for each $n\in\mathbb N$ and each ordered subset $(v_0,\ldots,v_n)\in V^{n+1}$ (the "$n$--simplices with vertices $v_0,\ldots,v_n$"),\\
\item[-] a map $\partial_i:S_{v_0,\ldots,v_n}\rightarrow S_{v_0,\ldots,\hat{v}_i,\ldots,v_n}$ for each $n\in\mathbb N$, each $i\in\left\{0,1,\ldots,n\right\}$ and each $(v_0,\ldots,v_n)\in V^{n+1}$.
\end{itemize}

For $n\in\mathbb N$ we will denote by $K_n$ the set of $n$--simplices of a multicomplex $K$, that is the union of all $S_{v_0,\ldots,v_n}$ over $(v_0,\ldots,v_n)\in V^{n+1}$.

A {\em simplicial map} from a multicomplex $K=(V_K,S,\left\{\partial_i\right\})$ to a multicomplex $K^\prime=(V_{K^\prime},S^\prime,{\{\partial^\prime_i\}})$ is given by a map $f:V_K\rightarrow V_{K^\prime}$ and a map $F:S_{v_0,\ldots,v_n}\rightarrow S^\prime_{f(v_0),\ldots,f(v_n)}$ for each $(v_0,\ldots,v_n)\in V_K^{n+1}, n\in\mathbb N$, such that $\partial_i^\prime F=F\partial_i$ for all $i\in\left\{0,1,\ldots,n\right\}$.

A {\em subcomplex} $K^\prime\subset K$ of a multicomplex $K=(V_K,S,\left\{\partial_i\right\})$ is a multicomplex $K^\prime=(V_{K^\prime},S^\prime,{\{\partial^\prime_i\}})$ with $V_{K^{\prime}}\subset V_K$, $S^\prime_{v_0,\ldots,v_n}\subset S_{v_0,\ldots,v_n}$ for all $(v_0,\ldots,v_n)\in V_{K^\prime}^{n+1}$, $n\in\mathbb N$, and $\partial_i^\prime:S^\prime_{v_0,\ldots,v_n}\rightarrow S^\prime_{v_0,\ldots,\hat{v}_i,\ldots,v_n}$ the restriction of $\partial_i:S_{v_0,\ldots,v_n}\rightarrow S_{v_0,\ldots,\hat{v}_i,\ldots,v_n}$ to $S^\prime_{v_0,\ldots,v_n}$ for all $(v_0,\ldots,v_n)\in V_{K^\prime}^{n+1}$, $n\in\mathbb N$, $i\in\left\{0,1,\ldots,n\right\}$.

The geometric realisation $\mid K\mid$ of a multicomplex $K=(V,S,{\{\partial_i\}})$ is $$\mid K\mid=\bigcup_{n\in \mathbb N}\bigcup_{(v_0,\ldots,v_n)\in V^{n+1}}S_{v_0,\ldots,v_n}\times\Delta^n/\sim,$$
where $\Delta^n$ denotes the standard $n$--simplex and the equivalence relation
$\sim$ is
generated by the relations $$(\partial_i\sigma,x)\sim(\sigma,j_i(x))$$
for all $n\in\mathbb N,(v_0,\ldots,v_n)\in V^{n+1},\sigma\in S_{v_0,\ldots,v_n},i\in\left\{0,\ldots,n\right\},x\in\Delta^{n-1}$, where $j_i:\Delta^{n-1}\rightarrow\Delta^n$ denotes the standard inclusion as $i$-th face.


For example the complex $\widehat{S}_*^{sing}\left(X\right)$ of singular simplices (with distinct vertices) in a given topological space $X$ is a multicomplex $K$ with $V=X$.

In this paper we will use the notation from \cite[Section 1]{Ku}, which is actually a collection of
the concepts and results from \cite[Section 3]{Gro82}. For the convenience of the reader we are now going to explain
shortly those facts about the bounded cohomology of multicomplexes which we will use in the proof of Theorem \ref{thm:2.2}. For more details we refer to \cite{Ku} or the original \cite{Gro82}.

A multicomplex $K$
is {\em minimally complete} if each singular simplex $$\sigma \co \Delta^n\rightarrow \mid K\mid$$ whose restriction to the boundary $f |_{\partial\Delta^n}:\partial\Delta^n\rightarrow\mid K\mid$ is a simplicial embedding, is homotopic rel.\ $\partial\Delta^n$ to a unique simplicial embedding.
$K$ is {\em aspherical} if simplices in $K$ are uniquely determined by their $1$--skeleton.

To each topological space $X$ one associates a
minimally complete multicomplex $$\widehat{K}\left(X\right)\subset \widehat{S}_*^{sing}\left(X\right)$$ as follows.
The vertices of $\widehat{K}\left(X\right)$ are the points of $X$, i.e., $\widehat{K}_0\left(X\right)=\hat{S}_0^{sing}(X)=X$. In each homotopy class
rel.\ $\partial\Delta^1$ of singular $1$--simplices $\Delta^1\rightarrow X$ (with two distinct vertices) we choose one $1$--simplex
to belong to $\widehat{K}_1\left(X\right)$. Inductively for $n\ge 2$ in
each homotopy class
rel.\ $\partial\Delta^n$ of singular $n$--simplices $\sigma \co \Delta^n\rightarrow X$ with $\partial_0\sigma,\ldots,\partial_n\sigma$ belonging to $\widehat{K}_{n-1}\left(X\right)$,
we choose one $n$--simplex
to belong to $\widehat{K}_n\left(X\right)$. We can and will succesively choose the $n$--simplices for $n=1,2,\ldots$ such that $\sigma\in\widehat{K}_n(X)\Leftrightarrow \sigma f_\pi\subset\widehat{K}_n(X)$ for any permutation $\pi\in S_{n+1}$, where $f_\pi:\Delta^n\rightarrow\Delta^n$ is the unique affine homeomorphism given by the permutation $\pi$ of the vertices of $\Delta^n$.

Next, an aspherical, minimally complete multicomplex $K\left(X\right)$ is defined as the quotient
of $\widehat{K}\left(X\right)$ under the
equivalence relation that identifies two simplices
whenever they have the same $1$--skeleton. Denoting $\Gamma_i$ the group of automorphisms fixing the $i$--skeleton we have $$K\left(X\right)=\Gamma_1\backslash \widehat{K}\left(X\right).$$
In \cite[Section 3.3, Corollary D]{Gro82} it is proved that $\Gamma_1/\Gamma_{i+1}$ is solvable and (hence) $\widehat{K}\left(X
\right)\rightarrow K\left(X\right)$ induces an isometric isomorphism of bounded cohomology $H_b^i$ for all $i$.

If $A\subset X$ is $\pi_1$--injective for each path component, then $K\left(A\right)$ is naturally a submulticomplex of $K\left(X\right)$. In general, if $A\subset X$ is not necessarily $\pi_1$--injective, then one has (non-injective) simplicial maps $$\hat{j} \co
\widehat{K}\left(A\right)\rightarrow \widehat{K}\left(X\right), \
j \co K\left(A\right)\rightarrow K\left(X\right),$$
see \cite[Section 1.3]{Ku}, and one considers the submulticomplex $L:=jK\left(A\right)\subset K\left(X\right)$. 
The argument of \cite[Section 3.3, Corollary D]{Gro82} clearly extends to the relative case thus for all $i$ one has an isometric isomorphism $$H_b^i\left(K\left(X\right),jK\left(A\right)\right)\rightarrow H_b^i
\left(\widehat{K}\left(X\right),\hat{j}\widehat{X}\left(A\right)\right).$$

{\bf Definition of $\Pi_X\left(A\right)$.} We recollect the definitions from \cite[Section 1.5.1]{Ku}, originally due to \cite{Gro82}. Let $\left(X,A\right)$ be a pair of spaces and $$L:=jK\left(A\right)\subset K\left(X\right).$$
We denote by $L_1$ the $1$--skeleton of $L$, by $\Omega L$ the
set of homotopy classes rel.\ $\left\{0,1\right\}$
of continuous maps $\gamma \co \left[0,1\right]\rightarrow \mid L\mid$ with $\gamma\left(0\right)=\gamma\left(1\right)$, and by $\Omega^*L\subset\Omega L$ the subset of nontrivial homotopy classes rel.\ $\left\{0,1\right\}$.
For $x\in A$ we denote by $c_x$ the constant loop based at $x$. Define $$
\Pi_X\left(A\right):=\left\{\begin{array}{c}
\gamma:A\rightarrow L_1\cup\Omega L: \\
\gamma(x)=c_x\ \mbox{for all but finitely many}\ x\in A,\\
\gamma(x)(0)=x\ \forall x, \ \gamma(x)(1)\not=\gamma(y)(1)\ \forall x\not=y,\\
\forall z\in A \ \exists x\in A\ \mbox{with}\ \gamma(x)(1)=z.
\end{array}\right\}$$
$$=\left\{\begin{array}{c}\left\{\gamma_1,\ldots,\gamma_n
\right\}
:n\in \mathbb N,
\gamma_1,\ldots,\gamma_n\in L_1\cup \Omega^* L,\\
\gamma_i\left(0\right)\not=\gamma_j\left(0\right), \ \gamma_i\left(1\right)\not=\gamma_j\left(1\right)\mbox{\ for\ }i\not= j,\\
\left\{\gamma_1\left(
0\right),\ldots,\gamma_n\left(0\right)\right\}=\left\{
\gamma_1\left(1\right),\ldots,\gamma_n\left(1\right)\right\}.
\end{array}\right\}.$$

The multiplication in $\Pi_X\left(A\right)$ is defined as follows:
given $\left\{\gamma_1,\ldots,\gamma_m\right\}$
and $\left\{\gamma_1^\prime,\ldots,\gamma_n^\prime\right\}$,
we choose a reindexing of the unordered sets $\left\{\gamma_1,\ldots,\gamma_m\right\}$
and $\left\{\gamma_1^\prime,\ldots,\gamma_n^\prime\right\}$
such that we have: $$\gamma_j\left(1\right)
=\gamma_j^\prime\left(0\right)$$ for $1\le j\le i$ and $$\gamma_j\left(1\right)
\not=\gamma_k^\prime\left(0\right)$$ for all pairs $\left(j,k\right)$ with
$j\ge i+1, k\ge i+1$.
(Since we are assuming that all $\gamma_j\left(1\right)$ are pairwise distinct, and also all
$\gamma_j^\prime\left(0\right)$ are pairwise distinct, such a reindexing exists for some $i\ge0$, and
it is unique up to permuting the indices $\le i$ and permuting separately the
indices of the $\gamma_j$ with $j\ge i+1$ and of the $\gamma_k^\prime$ with $k\ge i+1$.) \\
Moreover we permute the indices $\left\{1,\ldots,i\right\}$ such that there exists some $h$ with $0\le h\le i$ satisfying the following conditions:

\begin{itemize}
\item[-] for $1\le
j\le h$ we have either $\gamma_j^\prime\not=\overline{\gamma_j}\in L_1$ or $\gamma_j^\prime
\not=\gamma_j^{-1}\in\Omega^*L.$
\item[-] for $h<
j\le i$ we have either $\gamma_j^\prime=\overline{\gamma_j}\in L_1$ or $\gamma_j^\prime                                                    =\gamma_j^{-1}\in\Omega^*L,$
\end{itemize}
where $\overline{\gamma_j}$ denotes the unique $1$--simplex in the relative homotopy class of $t\rightarrow\gamma_j\left(1-t\right)$.

With this fixed reindexing we define
$$\left\{\gamma_1,\ldots,\gamma_m\right\}
\left\{\gamma_1^\prime,\ldots,\gamma_n^\prime\right\}:=
\left\{\gamma_1^\prime*\gamma_1,\ldots,\gamma_h^\prime*\gamma_h,\gamma_{i+1},\ldots,
\gamma_m,\gamma_{i+1}^\prime,
\ldots,\gamma_n^\prime\right\},$$
where $\gamma_j^\prime*\gamma_j$ denotes the unique element of $L_1\cup\Omega^*L$ homotopic rel.\ boundary to the path $l:\left[0,1\right]\rightarrow \mid L\mid$ defined by $l(t)=\gamma_j^\prime(2t)$ for $t\le\frac{1}{2}$ and $l(t)=\gamma_j(2t-1)$ for $t\ge\frac{1}{2}$.
Here we use the convention that $\gamma_j*\overline{\gamma_j}$ is 'empty'. (That means in this case the product is just avoided from the set on the right hand side.)\footnotemark\footnotetext[3]{Equivalently, but perhaps less intuitively, the multiplication is defined - using the first definition of $\Pi_X(A)$ - by $\gamma\gamma^\prime(x)=\gamma^\prime(x)*\gamma(\gamma^\prime(x)(1))$ for all $x\in A$.}

In particular the empty set is the neutral element of $\Pi_X\left(A\right)$ and the inverse of $\left\{\gamma_1,\ldots,\gamma_n\right\}$ is $\left\{\overline{\gamma_1},\ldots,\overline{\gamma_n}\right\}$.
We remark that $\gamma\gamma^\prime$ indeed belongs
to $\Pi_X\left(A\right)$ because $$\left\{ \gamma_1^\prime*\gamma_1\left(0\right),
\ldots,\gamma_h^\prime*\gamma_h\left(0\right),\gamma_{i+1}\left(0\right),\ldots,
\gamma_m\left(0\right),\gamma_{i+1}^\prime\left(0\right),
\ldots,\gamma_n^\prime\left(0\right)\right\}$$
$$=\left\{ \gamma_1^\prime*
\gamma_1\left(1\right),\ldots,\gamma_h^\prime*\gamma_h\left(1\right), \gamma_{i+1}\left(1\right),\ldots,
\gamma_m\left(1\right),\gamma_{i+1}^\prime\left(1\right),
\ldots,\gamma_n^\prime\left(1\right)\right\}.$$ This equality holds because the maps $\left(
\gamma_1^\prime\left(0\right),\ldots,\gamma_n^\prime\left(0\right)\right)\rightarrow
\left(\gamma_1^\prime\left(1\right),\ldots,\gamma_n^\prime\left(1\right)\right)$ and
$\left(\gamma_1\left(0\right),\ldots,\gamma_m\left(0\right)\right)\rightarrow\left(\gamma_1\left(1\right),\ldots,\gamma_m\left(1\right)\right)$ can be considered as permutations of $L_0=A$ keeping all but finitely many vertices fixed and hence the composition of the two permutations (as well as the composition of their inverses) will again be a permutation, in particular the composition and its inverse will both be injective.

It is well known that concatenation of paths defines an associative operation on the set of homotopy classes of paths rel.\ boundary. (Though concatenation is not associative on the set of paths.) Therefore $*$ is an associative operation on (a subset of) $L_1\cup\Omega L$. This implies associativity of the multiplication in $\Pi_X(A)$ because (using the first definition of $\Pi_X(A)$ to keep notation simpler):
{\setlength\arraycolsep{2pt}
\begin{eqnarray*}
(\gamma\gamma^\prime)\gamma^{\prime\prime}(x)&=&\gamma^{\prime\prime}(x)*\gamma^\prime(\gamma^{\prime\prime}(x)(1))*\gamma(\gamma^\prime(\gamma^{\prime\prime}(x)(1))(1)) \\
&=&\gamma^{\prime\prime}(x)*\gamma^\prime(\gamma^{\prime\prime}(x)(1))*\gamma(\gamma^{\prime\prime}(x)*\gamma^\prime(\gamma^{\prime\prime}(x)(1))(1))=\gamma(\gamma^\prime\gamma^{\prime\prime})(x)
\end{eqnarray*}}
for all $x\in A$. Thus we have turned $\Pi_X(A)$ into a group.\\

{\bf Action of $\Pi_X\left(A\right)$ on $K\left(X\right)$.} There is an
inclusion $$\Pi_X\left(A\right)\subset map_0\left(L_0,\left[\left[0,1\right],\mid L\mid\right]_{\mid K\mid}\right),$$
where $\left[\left[0,1\right],\mid L\mid\right]_{\mid K\mid}$ is the set of homotopy
classes (in $\mid K\mid:=\mid K\left(X\right)\mid$) rel.\ $\left\{0,1\right\}$ of maps from $\left[0,1\right]$ to $\mid L\mid$,
and $map_0\left(L_0,\left[\left[0,1\right],\mid L\mid\right]_{\mid K\mid}\right)$
is the set of maps $f \co L_0\rightarrow \left[\left[0,1\right],\mid L\mid\right]_{\mid K\mid}$ with
$f\left(y\right)\left(0\right)=y$ for all $y\in L_0$ and
$f\left(.\right)\left(1\right) \co L_0\rightarrow L_0$ is a bijection.
This inclusion is given by sending $$\left\{\gamma_1,\ldots,\gamma_n\right\}$$
to the map $f$ defined by $$f\left(\gamma_i\left(0\right)\right)=\left[\gamma_i\right]$$ for $i=1,\ldots,n$, and $$f\left(y\right)=\left[c_y\right]$$
(the constant path) for $y\not\in \left\{\gamma_1\left(0\right),\ldots,\gamma_n\left(0\right)\right\}$.

The inclusion is a homomorphism with respect to the multiplication defined on
$map_0\left(L_0,\left[\left[0,1\right],\mid L\mid\right]_{\mid K\mid}\right)$
by $\left[gf\left(y\right)\right]:=\left[f\left(y\right)\right]*\left[g\left(f\left(y\right)\left(1\right)\right)\right],$
where again $*$ denotes the well-defined concatenation of homotopy classes of paths rel.\ boundary.
An action of $$map_0\left(L_0,\left[\left[0,1\right],\mid L\mid\right]_K\right)$$ on $K$ is defined on the $0$--skeleton by $$gy=g\left(y\right)\left(1\right)$$ for $y\in L_0$ and $g\in map_0\left(L_0,\left[\left(0,1\right),\mid L\mid\right]_K\right)$, and by
$gx=x$ for $x\in K_0-L_0$.
To define the action on the $1$--skeleton of $K$ recall that, by minimal completeness of $K$, $1$--simplices $\sigma$ are in 1-1-correspondence with homotopy classes (rel.\ $\left\{0,1\right\}$)
of (nonclosed) singular $1$--simplices in $\mid K\mid$ with distinct vertices in $K_0$.

Using this correspondence,
define, for $\sigma\in K_1$, $g\sigma$ to be the unique $1$--simplex in the homotopy class
(rel.\ $\left\{0,1\right\}$) of $$\overline{g\left(\sigma\left(0\right)\right)}*\sigma*g\left(\sigma\left(1\right)\right).$$
The so defined $g\sigma$ is indeed an element of $K_1$
because $\overline{g\left(\sigma\left(0\right)\right)}*\sigma*g\left(\sigma\left(
1\right)\right)$ is a singular $1$--simplex with distinct vertices.
Indeed, if both vertices of $g\sigma$ agreed, then we would have
$g\left(\sigma\left(0\right)\right)\left(1\right)=g\left(\sigma\left(1\right)\right)\left(1\right)$.
But, since $g\left(\cdot \right)\left(1\right)$ is a bijection, this would contradict $\sigma\left(0\right)\not=\sigma\left(1\right)$.

We observe that the definition implies $g\overline{\sigma}=\overline{g\sigma}$ for all $g\in \pi_X(A),\sigma\in K_1$.
One checks easily that $\left(gf\right)\left(\sigma\right)=g\left(f\left(\sigma\right)\right)$ for all $$g,f\in map_0\left(L_0,\left[\left[0,1\right],\mid L\mid\right]_{\mid K\mid}\right), \ \sigma\in K_1.$$
Thus we defined an action of $map_0\left(L_0,\left[\left[0,1\right],\mid L\mid\right]_{\mid K\mid}\right)$ on $K_1$.

Observe that for a simplex
$\sigma\in K$ with $1$--skeleton $\sigma_1$, and $$g\in map_0\left(L_0,\left[\left[0,1\right],\mid L\mid\right]_{\mid K\mid}\right),$$
there exists {\em some} simplex in $K$ with $1$--skeleton $g\sigma_1$. Since $K$ is aspherical, this allows a unique extension of the group action from $K_1$ to $K$, see \cite[Section 1.5.1]{Ku}. Hence the action of $map_0\left(L_0,\left[\left[0,1\right],\mid L\mid\right]_K\right)$
extends (uniquely) to a simplicial action on $K$. In particular this defines an action of $\Pi_X\left(A\right)$ on $K$.

Since $g\overline{\sigma}=\overline{g\sigma}$ on the $1$--skeleton and since we assumed that $\sigma\in\widehat{K}(X)$ implies $\sigma\circ f_\pi\in\widehat{K}(X)$ for all permutations of vertices $\pi$, the construction implies $g(\sigma\circ f_\pi)=(g\sigma)\circ f_\pi$ for all permutations of vertices, in particular $g\overline{\sigma}=\overline{g\sigma}$ for all $g\in\Pi_X(A),\sigma\in \widehat{K}(X)$.

\subsection{An application of the relative mapping theorem.}\label{sec:2.2}
We stick to the notation of Section \ref{sec:2.1}. If each path component of
$A$ is a $\pi_1$--injective subspace of $X$, then $K\left(A\right)$
is a submulticomplex of $K\left(X\right)$ and by \cite[Proposition 3]{Ku}
one has a natural isometric isomorphism $$I^* \co H_b^*\left(K\left(X\right),K\left(A\right)\right)\rightarrow H_b^*\left(X,A\right).$$
Here the relative bounded cohomology of a pair of multicomplexes is by definition the simplicial bounded cohomology, that is the cohomology computed from the dual of the $\ell^1$--completion of the relative simplicial chain complex.

If the assumption on $\pi_1$--injectivity is not satisfied, then one has the submulticomplex $jK\left(A\right)\subset K\left(X\right)$ but in general it is not true that $H_b^*\left(K\left(X\right),jK\left(A\right)\right)$ is isomorphic to $H_b^*\left(X,A\right)$. However we are going to prove (as a consequence of Gromov's Relative Mapping Theorem) that one has an isometric isomorphism if all components of $A$ have amenable fundamental group.

As a preparation we need the following lemma.
\begin{lemma}\label{lem:fundamentalgroup}Let $(X,A)$ be a pair of topological spaces. Then
$$\pi_1(\mid\hat{j}\widehat{K}(A)\mid,x)\cong im(\pi_1(A,x)\rightarrow \pi_1(X,x))$$
for all $x\in A$ and the homomorphism of fundamental groups induced by the continuous mapping $\mid S_*^{sing}(A)\mid\rightarrow \mid \hat{j}\widehat{K}(A)\mid$ equals the homomorphism from $\pi_1(A,x)$ to $im(\pi_1(A,x)\rightarrow\pi_1(X,x)$.\end{lemma}
\begin{proof}Consider the commutative diagram
$$\begin{xy}
\xymatrix{
\mid\widehat{K}(A)\mid\ar[r]^{\hat{j}}\ar[d]^{p_A}&\mid\hat{j}\widehat{K}(A)\mid\ar[r]^{\hat{i}}&\mid\widehat{K}(X)\mid\ar[d]^{p_X}\\
A\ar[rr]^i& & X}
\end{xy} $$
We want to show that $(p_X \circ \hat{i})_*:\pi_1(\mid\hat{j}\widehat{K}(A)\mid,x)\rightarrow\pi_1(X,x)$ is injective and has image $im(\pi_1(A,x)\rightarrow\pi_1(X,x))$.

Let $\gamma:(S^1,o)\rightarrow (A,x)$ represent some non-zero element in $\pi_1(A,x)$. We think of $S^1$ as a union of two $1$--simplices $S^1=I_1\cup I_2$ with one of their common vertices in the base point $o$. By definition of $\widehat{K}(A)$ we can homotope $\gamma$, keeping the base point fixed, such that $\gamma(I_1)$ and $\gamma(I_2)$ become $1$--simplices in $\widehat{K}_1(A)$. Then $(p_X\circ \hat{i} \circ \hat{j})_*(I_1*I_2)$ represents $i_*(\gamma)$, proving that $im(i_*)$ is contained in $im(p_X \circ \hat{i})_*$. The reverse inclusion follows from commutativity of the diagram and surjectivity of $\hat{j}_*$.

To prove injectivity
of $(p_X \circ \hat{i})_*$
we will use relative simplicial approximation.
The absolute simplicial approximation theorem states that there is (after a sufficiently fine barycentric subdivision) a simplicicial approximation $g$ to a given continuous map $f$ between two simplicial complexes. The relative theorem permits one to leave $f$ unchanged on any subcomplex on which $f$ happens to already be simplicial, see \cite[Theorem 1.6.11]{Rush73}. Its proof in \cite{Rush73} does not make any use of the fact that simplices are uniquely determined by their vertices, hence it can also be applied to multicomplexes.

Now assume that a map $\gamma:(S^1,o)\rightarrow (\mid\hat{j}\widehat{K}(A)\mid,x)$ represents $\left[\gamma\right]\in ker(p_X \circ \hat{i})_*\subset\pi_1(\mid\hat{j}\widehat{K}(A)\mid,x)$. After fixing a sufficiently fine simplicial subdivision $S^1=I_1\cup\cdots\cup I_k$ we can assume the map to be simplicial, that is, the $1$--simplices $\gamma |_{I_1},\ldots,\gamma |_{I_k}$ belong to $\hat{j}\widehat{K}_1(A)$. Let $\gamma_1,\ldots,\gamma_n\in\widehat{K}_1(A)$ such that $\hat{j}(\gamma_i)=\gamma |_{I_i}$ for $i=1,\ldots,n$.

By assumption $(p_X \circ \hat{i}\circ \hat{j})
(\gamma_1*\gamma_2)
$ represents $0\in\pi_1(X,x)$, hence it bounds a map $f:(\mathbb D^2,o)\rightarrow(X,x)$. We subdivide $\mathbb D^2$ into $k-2$ triangles\footnotemark\footnotetext[4]{We can w.l.o.g.\ assume $k\ge 4$ to have a unified argument.} $D_1,\ldots,D_{k-2}$ such that $D_1$ is bounded by the edges $I_1,I_2$ and another edge $J_1$, for $2\le i\le k-3$ $D_i$ is bounded by the edges $I_{i+1},J_{i-1}$ and another edge $J_i$, and $D_{k-2}$ is bounded by the edges $I_{k-1},I_k,J_{k-3}$.

By definition of $\widehat{K}(X)$ we can homotope $f|_{D_1}$ (keeping $f|_{I_1}$ and $f|_{I_2}$ fixed) into some $2$--simplex $T_1\in \widehat{K}_2(X)$ which is bounded by 
$\hat{j}(\gamma_1),\hat{j}(\gamma_2)$
and the unique $1$--simplex $\tau_1\in \widehat{K}_1(X)$ in the homotopy class rel.\ boundary of
$f |_{J_1}$.
We observe that there exists some (degenerate) $2$-simplex $s_1:\Delta^2\rightarrow A$ whose boundary consists of $\gamma |_{I_1},\gamma |_{I_2}$ and $\gamma |_{I_1}*\gamma |_{I_2}$, where the latter means the unique $1$-simplex in $\widehat{K}_1(A)$ homotopic rel.\ boundary to the concatenation of $\gamma |_{I_1}$ and $\gamma |_{I_2}$. Denote $S_1\in \widehat{K}_2(A)$  the unique $2$-simplex in $\widehat{K}_2(A)$ homotopic rel.\ boundary to $s_1$. 
Since simplices in $\widehat{K}(X)$ are uniquely determined by their homotopy class rel.\ boundary, we necessarily have $\hat{j}(\gamma |_{I_1}*\gamma |_{I_2})=\tau_1$ and $\hat{j}(S_1)=T_1$. In particular $T_1\in\hat{j}\widehat{K}_1(A)$ which implies that the concatenation of $\gamma |_{I_1}, \gamma |_{I_2}$ and $\tau_1$ represents $0\in\pi_1(\mid\hat{j}\widehat{K}(A)\mid,x)$.

Successive application of the same argument to $D_2,\ldots,D_{k-3}$ yields $T_i\in \hat{j}\widehat{K}_2(A)$ and  $\tau_i\in\hat{j}\widehat{K}_1(A)$ such that the concatenation of $\overline{\tau}_{i-1}, \gamma |_{I_{i+1}}$ and $\tau_i$ represents $0\in\pi_1(\mid\hat{j}\widehat{K}(A)\mid,x)$
for all $i=2,\ldots,k-3$. Finally, application of the argument to $D_{k-2}$ 
shows that
the concatenation of $\overline{\tau}_{k-3}, \gamma |_{I_{k-1}}$ and
$\gamma |_{I_k}$ represents $0\in\pi_1(\mid\hat{j}\widehat{K}(A)\mid,x)$.

This implies that $\gamma |_{I_1}*\cdots*\gamma |_{I_{k}}$ represents $0\in\pi_1(\hat{j}\widehat{K}(A),x)$ which we wanted to prove.

The second claim of Lemma \ref{lem:fundamentalgroup} then follows from the commutative diagram in the beginning of the proof together with the fact that inclusion $\widehat{K}(A)\subset S_*^{sing}(A)$ induces an isomorphism $\pi_1(\mid\widehat{K}(A)\mid,x)\rightarrow\pi_1( \mid S_*^{sing}(A)\mid,x) \cong\pi_1(A,x)$ by \cite[page 42]{Gro82}.
\end{proof}
Now we are ready to handle the case when $\pi_1A\rightarrow\pi_1X$ is not injective.
\begin{lemma}\label{lem:2.1}
Let $(X,A)$ be a pair of topological spaces. Assume that each path
component of $A$ has amenable fundamental group.
Then there is an isometric isomorphism $$I^* \co H_b^*\left(K\left(X\right),jK\left(A\right)\right)\rightarrow H_b^*\left(X,A\right).$$
\end{lemma}

\begin{proof}By the isometry lemma in \cite[page 43]{Gro82} and its relative version in \cite[Proposition 1]{Ku} there is an
isometric isomorphism
$$H_b^*\left(X,A\right)=H_b^{simp}\left(S_*^{sing}\left(X\right),S_*^{sing}\left(A\right)\right).$$
Moreover (see Section \ref{sec:2.1}) there is an isometric isomorphism
$$H_b^{simp}\left(K\left(X\right),jK\left(A\right)\right)\rightarrow H_b^{simp}
\left(\widehat{K}\left(X\right),\hat{j}\widehat{K}\left(A\right)\right).$$
Thus it suffices to show that the inclusion $$\left(\widehat{K}\left(X\right),\hat{j}\widehat{K}\left(A\right)\right)\subset
\left(S_*^{sing}\left(X\right),S_*^{sing}\left(A\right)\right)$$ induces an isometric isomorphism in (simplicial)
bounded cohomology.

The inclusion $S \co \widehat{K}\left(X\right)\rightarrow S_*^{sing}\left(X\right)$ has a (simplicial) chain homotopy inverse $T$ as in \cite[Section 9]{May92}. By construction it maps $S_*^{sing}\left(A\right)$ to $\hat{j}\widehat{K}\left(A\right)$. (One should be aware that the restriction of $T$
to $S_*^{sing}\left(A\right)$ is not a chain homotopy equivalence if $\pi_1A$ does not inject into $\pi_1X$.) The simplicial map $T$ induces a continuous map $$T \co
\left(\mid S_*^{sing}\left(X\right)\mid, \mid S_*^{sing}\left(A\right)\mid\right)\rightarrow \left(\mid\widehat{K}\left(X\right)\mid,\mid
\hat{j}\widehat{K}\left(A\right)\mid\right).$$
The map $T$ induces an isomorphism of path components and for each path component the induced homomorphism of
fundamental groups is surjective with amenable kernel. (Indeed by Lemma \ref{lem:fundamentalgroup} the homomorphism of fundamental groups based at $x$ is the homomorphism from $\pi_1\left(A,x\right)$ to $im\left(\pi_1
\left(A,x\right)\rightarrow\pi_1\left(X,x\right)\right)$, whose kernel is a subgroup of $\pi_1\left(A,x\right)$, hence amenable.)

Thus the assumptions of Gromov's Relative Mapping Theorem (\cite[Page 57]{Gro82}) are satisfied
and we obtain an isometric isomorphism in (singular)
bounded cohomology
$$T^* \co H_b^*\left(\mid\widehat{K}\left(X\right)\mid,\mid
\hat{j}\widehat{K}\left(A\right)\mid\right)\rightarrow
H_b^*\left(\mid S_*^{sing}\left(X\right)\mid, \mid S_*^{sing}\left(A\right)\mid\right).$$
Again by the isometry lemma this implies that we also have an isometric isomorphism in simplicial bounded cohomology, hence the claim of the lemma.
\end{proof}

\subsection{Proof of Isometry.}\label{sec:2.3}

\begin{theorem}\label{thm:2.2}
Let $(X,A)$ be a pair of topological spaces. Assume that each path
component of $A$ has amenable fundamental group.
Then the inclusion $\left(X,\emptyset\right)\rightarrow \left(X,A\right)$ induces an isometric isomorphism
$$H_b^*(X,A)\rightarrow H_b^*(X)$$
in degrees $*\ge 2$.
\end{theorem}

\begin{remark} It is well known (\cite{Gro82}, \cite{Iva87}) that amenability of $\pi_1A$ implies $H_b^n\left(A\right)=0$ for $n\ge 1$ and hence the algebraic isomorphism $H_b^*\left(
X,A\right)\cong H_b^*\left(X\right)$ for $*\ge 2$ follows from the long exact sequence associated to the pair $\left(X,A\right)$. Moreover it is clear by construction that this isomorphism has norm $\le 1$. The nontrivial part of Theorem \ref{thm:2.2} is to prove that the isomorphism is indeed an isometry.\end{remark}
\begin{proof}
The proof builds on Gromov's theory of multicomplexes developed in \cite[Section 3]{Gro82}, but we will stick to the notation from \cite[Section 1]{Ku}.

As in Section \ref{sec:2.1}, to the topological space $X$ we have associated the (minimally complete, aspherical) multicomplex $K\left(X\right)$ and an
isometric isomorphism $$ H_b^*\left(K\left(X\right)\right)\rightarrow H_b^*\left(X\right)$$ from the simplicial bounded cohomology of $K\left(X\right)$ to the singular bounded cohomology of $X$, see \cite[page 45--46]{Gro82}.

If each path component of
$A$ is amenable, then we have from
Lemma \ref{lem:2.1} the isometric isomorphism
$$I^* \co H_b^*\left(K\left(X\right),jK\left(A\right)\right)\rightarrow H_b^*\left(X,A\right).$$
Thus the commutative diagram
$$\begin{xy}
\xymatrix{
H_b^*\left(X\right)&H_b^*\left(X,A\right)\ar[l]_-{i_1^*}\\
 H_b^*\left(K\left(X\right)\right)\
\ar[u]^-{I^*}&H_b^*\left(K\left(X\right),jK(A)\right) \ar[u]_-{I^*}\ar[l]_-{i_2^*}}
\end{xy} $$
implies that to prove the theorem it is sufficient to prove that $$H_b^*\left(K\left(X\right),jK\left(A\right)\right)\rightarrow H_b^*\left(K\left(X\right)\right)$$ is an isometric isomorphism in degrees $*\ge 2$.

Let $G=\Pi_X(A)$ be the group defined in Section \ref{sec:2.1} with its action on $K\left(X\right)$.
Amenability of $\pi_1A$ for all path-components of $A$ implies that $G$ is amenable by
\cite[Lemma 4]{Ku}.

Consider the following commutative diagram, where simplicial chains are by
definition alternating, i.e.\ for each simplex $\sigma$ and each permutation $\pi$ of its vertices we have
$\sigma\circ f_\pi=sign(\pi)\sigma$ in $C_*^{simp}$, in particular $\overline{\sigma}=-\sigma$ for
$\overline{\sigma}$ meaning $\sigma$ with the opposite orientation. The vertical arrows
$k_1,k_2$ are the obvious chain homomorphisms mapping $\sigma$ to $\sigma\otimes 1$.
$$\begin{xy}
\xymatrix{
C_*^{simp}\left(K\left(X\right)\right)
\ar[r]^-{i_2}\ar[d]_-{k_1}&C_*^{simp}\left(K\left(X\right),jK(A)\right)
\ar[d]^-{k_2}\\
 C_*^{simp}\left(K\left(X\right)\right)\otimes_{{\mathbb Z}G}{\mathbb Z}
\ar[r]^-{i_3}&C_*^{simp}\left(K\left(X\right),jK\left(A\right)\right)\otimes_{{\mathbb
Z}G} {\mathbb Z} }\end{xy} $$ Here $\mathbb Z$ is considered an $\mathbb ZG$--module with the trivial
$G$--action, i.e.\ taking the tensor product $\otimes_{{\mathbb Z}G}{\mathbb Z}$ means quotienting out the $G$--action. We remark that the tensor products in the lower row are defined because $g(\sigma\circ f_\pi)=(g\sigma)\circ f\pi$ for all permutations of vertices. Moreover $G=\Pi_X(A)$ maps $jK(A)$ to itself: this follows easily from the observation that the action of $\Pi_A(A)$ on $\widehat{K}(A)$ is compatible with the action of $\Pi_X(A)$ on $\widehat{K}(X)$ (with respect to the quotient maps $\Pi_A(A)\rightarrow\Pi_X(A)$ and $\hat{j}:\widehat{K}(A)\rightarrow \widehat{K}(X)$).

On $C_*^{simp}\left(K\left(X\right)\right)\otimes_{{\mathbb Z}G}{\mathbb Z}$ we have the $l^1$--seminorm
$$\left\| \sum_{i=1}^r(\sum_{j=1}^sa_{ij}\sigma_{ij})\otimes n_i \right\| _1=\sum_{i=1}^r\mid n_i\mid\sum_{j=1}^s\mid a_{ij}\mid$$
for $\sum_{j=1}^sa_{ij}\sigma_{ij}\in C_*^{simp}\left(K\left(X\right)\right)$ and $n_i\in\mathbb Z$. (This is well-defined, independently of the chosen representative in the tensor product.) On the complex of bounded cochains we consider the dual $l^\infty$--seminorm. Then the complex of bounded cochains $C_b^*\left(K\left(X\right)\otimes_{{\mathbb Z}G}{\mathbb Z}\right)$ agrees with the complex of $G$--invariant bounded simplicial cochains on $K\left(X\right)$.

Accordingly, on $C_*^{simp}\left(K\left(X\right),jK\left(A\right)\right)\otimes_{{\mathbb
Z}G} {\mathbb Z}$ we define the $l^1$--seminorm by
$$\| w\|_1=\inf\left\{\| z\|_1: z\in C_*^{simp}\left(K\left(X\right)\right)\otimes_{{\mathbb Z}G}{\mathbb Z}, i_3(z)=w\right\}$$ and on the complex of bounded cochains we consider the dual $l^\infty$--seminorm. We remark that $C_b^*\left(C_*(K\left(X\right),jK(A))\otimes_{{\mathbb Z}G}{\mathbb Z}\right)$ agrees with the complex of $G$--invariant bounded cochains on $K\left(X\right)$ which vanish on $jK(A)$.

Since $G$ is amenable, we have on the level of bounded cochains a left-inverse $av$  to $k_1^*$, which is given by averaging bounded cochains, thus mapping bounded cochains to $G$--invariant bounded cochains. Explicitly, $av$ is given as follows: for $f\in C_b^*(K(X))$ and $\sigma\otimes 1\in C_*^{simp}\left(K\left(X\right)\right)\otimes_{{\mathbb Z}G}{\mathbb Z}$ we equivariantly identify $G\sigma$ with a quotient of $G$, use this identification to pull $f$ back to a bounded function $\tilde{f}$ on $G$ and then use the $G$--invariant mean on $G$ - which exists by definition of amenability - to define $av(f)(\sigma)$ to be the mean of $\tilde{f}$. By the Proposition on \cite[page 48]{Gro82} (see also \cite[Theorem 2(ii)]{Ku}),
which applies because all elements of $G=\Pi_X\left(A\right)$ are homotopic to the identity by construction (see \cite[Section 1.5]{Ku} for the latter fact)
it follows that $k_1^*$ is an isometric isomorphism is bounded cohomology.
The same argument yields that also $k_2^*$ is an isometric isomorphisms in bounded cohomology.

Thus in the commutative diagram
$$\begin{xy}
\xymatrix{
H_b^*\left(K\left(X\right)\right)&H_b^*\left(K\left(X\right),jK(A)\right)\ar[l]_-{i_2^*}\\
 H_b^*\left(C_*^{simp}\left(K\left(X\right)\right)\otimes_{{\mathbb Z}G}{\mathbb Z}\right)
\ar[u]^-{k_1^*}&H_b^*\left(C_*^{simp}\left(K\left(X\right),jK(A)\right)\otimes_{{
\mathbb Z}G} {\mathbb Z}\right) \ar[u]_-{k_2^*}\ar[l]_-{i_3^*}}
\end{xy} $$
we have that $k_1^*$ and $k_2^*$ are isometric isomorphisms.
So to prove the claim it suffices to show that $i_3$ induces
an isometric isomorphism in bounded cohomology (in degrees $*\ge 2$).
We will prove that $i_3$,
on the level of chain or (bounded) cochain groups, is indeed even a group isomorphism of chain groups in degrees $*\ge 1$. This will imply that we have an isomorphism in bounded cohomology in degrees $*\ge2$.\\

{\bf Subclaim}: {\em If $*\ge 1$, then }
$$i_3 \co C_*^{simp}\left(K\left(X\right)\right)\otimes_{{\mathbb Z}G}{\mathbb Z}
\rightarrow C_*^{simp}\left(K\left(X\right),jK(A)\right)\otimes_{{\mathbb
Z}G} {\mathbb Z} $$ {\em is an isomorphism.}\\

{\bf Proof of Subclaim}:\\
\\
Taking tensor products is a right-exact functor, thus we have an exact sequence
$$C_*^{simp}\left(jK\left(A\right)\right)\otimes_{\mathbb ZG}\mathbb Z\rightarrow
C_*^{simp}\left(K\left(X\right)\right)\otimes_{\mathbb ZG}\mathbb Z\rightarrow
C_*^{simp}\left(K\left(X\right),jK\left(A\right)\right)\otimes_{\mathbb ZG}\mathbb Z\rightarrow 0.$$
Thus it suffices to prove that
$$C_*^{simp}\left(jK\left(A\right)\right)\otimes_{\mathbb ZG}\mathbb Z=0$$
for $*\ge 1$.

Let $\sigma$ be a simplex
in $jK(A)$. Since the $G$--action on
$\mathbb Z$ is trivial we have $\sigma\otimes 1=g\sigma\otimes 1$ in
$C_*^{simp}\left(jK\left(A\right)\right)\otimes_{\mathbb ZG}\mathbb Z\subset C_*^{simp}\left(K\left(X\right)\right)\otimes_{{\mathbb Z}G}{\mathbb Z}$.

By \cite[Observation 1]{Ku} the following is true: \\

{\em for
each simplex $\sigma$ with at least one edge in $A$ there is some
$g\in G$ with $g\sigma=\overline{\sigma}$.} \\
\\
In particular this applies to simplices of dimension $*\ge 1$ (but not to $0$--simplices!) in $jK\left(A\right)$.
As a consequence, for each simplex of dimension $\ge 1$ in $jK(A)$ we have
$$\sigma\otimes 1=g\sigma\otimes 1=\overline{\sigma}\otimes
1=-\sigma\otimes 1,$$ hence $$\sigma\otimes 1=0$$ in
$C_*^{simp}\left(jK\left(A\right)\right)\otimes_{\mathbb ZG}\mathbb Z\subset C_*^{simp}\left(K\left(X\right)\right)\otimes_{{\mathbb Z}G}{\mathbb Z}$. This finishes the proof of the subclaim.\\

From the subclaim it follows easily that $i_3$ induces an isometric isomorphism in bounded cohomology in degrees $*\ge 2$. Recall that the complex defining simplicial bounded cohomology is defined as the dual of the $\ell^1$--completion of the simplicial chain complex. The $\ell^1$--norm on $$C_*^{simp}\left(K\left(X\right),jK(A)\right)\otimes_{{\mathbb Z}G}{\mathbb Z}$$
is defined as the infimum of the $\ell^1$--norms of the preimages under $i_3$. In particular $i_3$ being an isomorphism for $*\ge1$ implies that it is an isometry in that range. Now consider the chain complexes $C_*,D_*$ defined by
$$C_0=0, \ C_*=C_*^{simp}\left(K\left(X\right)\right)\otimes_{{\mathbb Z}G}{\mathbb Z}\mbox{\ for\ }*\ge 1$$
$$D_0=0, \ D_*=C_*^{simp}\left(K\left(X\right),jK(A)\right)\otimes_{{\mathbb Z}G}{\mathbb Z}\mbox{\ for\ }*\ge 1$$
with the obvious boundary operators. We have just proved that $$i_3 \co C_*\rightarrow D_*$$ is an isometric isomorphism for the $\ell^1$--norm, which implies that $\left(i_3\right)^*_b$ is also an isometry for the bounded cochain complexes (the duals of the $\ell^1$--completions).

Since the bounded cochain complexes of $C_*$ and $C_*^{simp}\left(K\left(X\right)\right)\otimes_{{\mathbb Z}G}{\mathbb Z}$ agree in degrees $\ge 1$, their cohomology (and its induced pseudonorm) agree in degrees $\ge 2$. In the same way in degrees $\ge 2$ the bounded cohomology of $D_*$ agrees with that of $C_*^{simp}\left(K\left(X\right),jK(A)\right)\otimes_{{\mathbb Z}G}{\mathbb Z}$. Thus we have proved that $\left(i_3\right)_b^*$ is an isometric isomorphism in degrees $*\ge 2$, which finishes the proof of Theorem \ref{thm:2.2}.
\end{proof}

\section{Simplicial volume}

In this section, we devote ourselves to verifying Theorem \ref{thm:1.3}.
We first collect some definitions and results on the simplicial volume mainly due to Gromov \cite{Gro82}. In the sequel homology will always be understood to be homology with $\mathbb R$-coefficients.

\subsection{Simplicial $\ell^1$--norm and simplicial volume}\label{sec:3.1}
Let $X$ be a topological space. Let $C_*(X)$ denote the real singular chain complex of $X$. Given a $k$--chain $c=\sum_{i=1}^r a_i \sigma_i$ where $\sigma_1,\ldots,\sigma_r$ are distinct singular simplices, the simplicial $\ell^1$--norm $\| c\|_1$ of $c$ is defined by $\|c \|_1 = \sum_{i=1}^r |a_i |.$ The simplicial $\ell^1$--norm induces an $\ell^1$--seminorm on the singular homology $H_*(X)$ as follows: For a singular homology class $\alpha \in H_k(X)$, the $\ell^1$--seminorm $\| \alpha \|$ is defined by
$$\| \alpha \| = \inf \{ \|c\|_1 \ | \ c \in C_k(X) \text{ cycle representing }\alpha \}.$$

For a connected, oriented, closed manifold $M$, \emph{the simplicial volume $\| M \|$ of $M$} is defined as the $\ell^1$--seminorm of its fundamental class. If $M$ is not orientable, the simplicial volume $\| M \|$ is defined by $\| M\|= \frac{1}{2} \|\widetilde{M} \| $ where $\widetilde{M}$ is the oriented double covering of $M$.

In the same way, the simplicial volume $\|M\|$ of an open manifold $M$ is defined via locally finite homology instead of singular homology. This is because the fundamental class of $M$ lives in the locally finite homology of $M$. More precisely, let $M$ be a connected, oriented $n$--manifold without boundary. Let $c=\sum_{i=1}^\infty a_i \sigma_i$ be a formal infinite linear sum of singular $k$--simplices in $M$ with coefficients in $\mathbb{R}$.
The infinite chain $c$ is said to be \textit{locally finite} if any compact subset of $M$ intersects the image of only finitely many singular simplices occurring in $c$. The \emph{locally finite homology $H^\mathrm{lf}_*(M)$ of $M$} is defined as the homology of the complex $C^\mathrm{lf}_*(M)$ consisting of locally finite chains.
It is a standard fact that there is a fundamental class $[M]$ in $H_n^\mathrm{lf}(M)$ (See \cite{BM60}). Note that one advantage of locally finite homology is the existence of a fundamental class of any oriented manifold.

For a locally finite chain $c=\sum_{i=1}^\infty a_i \sigma_i$ in $M$, the $\ell^1$--norm $\|c\|_1$ of $c$ is defined by $\|c\|_1 =\sum_{i=1}^\infty |a_i|$.
This $\ell^1$--norm gives rise to an $\ell^1$--seminorm on the locally finite homology $H^\mathrm{lf}_*(M)$ in the same way as the $\ell^1$--seminorm on the singular homology of $M$ is defined. \emph{The simplicial volume $\|M\|$ of $M$} is defined as the $\ell^1$--seminorm of the fundamental class in $H^\mathrm{lf}_n(M)$, that is,
$$\| M \| = \inf \{ \|c\|_1 \ | \ c \in C^\mathrm{lf}_n(M) \text{ fundamental cycle of }M \}.$$
Note that the $\ell^1$--norm of a locally finite chain is not necessarily finite and hence, the simplicial volume of open manifolds can be infinite.

For a compact $n$--manifold $V$ with boundary $\partial V$, the simplicial $\ell^1$--norm of each element in the real relative singular chain complex $C_*(V,\partial V)$ is defined as the infimum of the $\ell^1$--norms of its representatives and gives rise to a seminorm on the relative singular homology $H_*(V,\partial V)$. \emph{The relative simplicial volume $\|V,\partial V\|$} is defined as the seminorm of the relative fundamental class $[V,\partial V] \in H_n(V,\partial V)$. We refer the reader to \cite{Gro82} for further details on the simplicial volume.

\subsection{Geometric simplicial volume}\label{sec:3.2}
Gromov introduced another kind of simplicial volume for Riemannian open manifolds in \cite[Section 4.4]{Gro82}. Let $(M,g_M)$ be a connected, oriented, open, Riemannian $n$--manifold.
Fixing a metric on the standard $k$--simplex $\Delta^k$ by the Euclidean metric, the Lipschitz constant $\mathrm{Lip}(\sigma)$ of a singular simplex $\sigma : \Delta^k \rightarrow M$ is defined with respect to the Riemannian metric on $M$. For each locally finite $k$--chain $c \in C^\mathrm{lf}_k(M)$, the Lipschitz constant $\mathrm{Lip}(c)$ is defined as the supremum of all Lipschitz constants of the simplices occurring in $c$.
Then, the geometric simplicial volume {$\|M,g_M\|_\mathrm{Lip}$} is defined by
$${\|M,g_M\|_\mathrm{Lip}} =\{ \|c \|_1 \ | \ c \in C^\mathrm{lf}_n(M)\text{ fundamental cycle with }\mathrm{Lip}(c)<\infty \}.$$
{Note that the geometric simplicial volume of $M$ depends on the Riemannian metric $g_M$.

{
In the case that $M$ is a closed Riemannian manifold, its geometric simplicial volume is equal to its simplicial volume and therefore independent of the Riemannian metric.}
Gromov \cite{Gro82} establishes the proportionality principle for the geometric simplicial volume.
\begin{theorem}[Gromov \cite{Gro82}]\label{thm:3.1}
Let {$(M,g_M)$} be a closed Riemannian manifold, and let {$(N,g_N)$} be a complete Riemannian manifold of finite volume. Assume the universal covers of $M$ and $N$ are isometric. Then,
{$$\frac{\| M,g_M \|_\mathrm{Lip}}{\mathrm{Vol}(M,g_M)}= \frac{\| N,g_N \|_\mathrm{Lip}}{\mathrm{Vol}(N,g_N)}.$$}
\end{theorem}

Under a nonpositive curvature condition, L\"{o}h and Sauer \cite{LS09-2} also establish the following proportionality principle.
\begin{theorem}[L\"{o}h and Sauer \cite{LS09-2}]\label{thm:3.2}
Let {$(M,g_M)$} and {$(N,g_N)$} be complete, nonpositively curved Riemannian manifolds of finite volume whose universal covers are isometric. Then
{$$\frac{\| M,g_M \|_\mathrm{Lip}}{\mathrm{Vol}(M,g_M)}= \frac{\| N,g_N \|_\mathrm{Lip}}{\mathrm{Vol}(N,g_N)}.$$}
\end{theorem}

Note that the proportionality principle for the ordinary simplicial volume of open manifolds fails in general.
There are locally symmetric spaces of noncompact type with $\mathbb{Q}$--rank $3$ and $\mathbb{Q}$--rank $1$ whose universal covers are isometric, but one has positive simplicial volume and the other has vanishing simplicial volume. {Here is one such example. Let $M$ be a Hilbert modular variety and $N$ be the Cartesian product of three noncompact hyperbolic surfaces, both of which are covered by $\mathbb{H}^2 \times \mathbb{H}^2 \times \mathbb{H}^2$. Then $\| M\|$ is positive \cite{LS09-1} and $\|N\|=0$ \cite[Section 4.2]{Gro82}.} Hence the proportionality principle does not work. 

{As a corollary of Theorem \ref{thm:3.2}, it can be seen that the geometric simplicial volume of locally symmetric spaces of finite volume and noncompact type is nonzero \cite{LS09-2}. This follows from the fact that the simplicial volume of closed locally symmetric spaces of noncompact type is nonzero, which is proved by Lafont and Schmidt \cite{LS06}. The presence of the proportionality principle is one advantage of the geometric simplicial volume against the ordinary simplicial volume.}

\subsection{Bounded cohomology and $\ell^1$--homology}\label{sec:3.3}
{Gromov in his paper \cite{Gro82} defined the bounded cohomology of topological spaces and proved a number of theorems about it. Let $S_k(X)$ denote the set of all singular $k$--simplices of a topological space $X$.} Then, for a singular $k$--cochain $f \co C_k(X) \rightarrow \mathbb{R}$, the $\ell^\infty$--norm $\|f\|_\infty$ of $f$ is defined by setting $$\| f \|_\infty = \sup \{ |f(\sigma)| \ | \ \sigma \in S_k(X)\}.$$

{A singular cochain is said to be \emph{bounded} if its $\ell^\infty$--norm is finite, i.e., its set of values on singular simplices is bounded.}
Let $C^k_b(X)$ be the set consisting of bounded singular $k$--cochains. It is easy to see that $C^*_b(X)$ is preserved by the coboundary operator. Hence, $C^*_b(X)$ is a cochain complex and gives rise to a cohomology. This cohomology is called \emph{the bounded cohomology of $X$}, denoted by $H^*_b(X)$. The $\ell^\infty$--norm induces an $\ell^\infty$--seminorm on the bounded cohomology of $X$ as follows: For a bounded cohomology class $\varphi \in H^k_b(X)$, the $\ell^\infty$--seminorm $\| \varphi \|_\infty$ of $\varphi$ is defined by
$$\| \varphi \|_\infty = \inf \{ \|f\|_\infty \ | \ f \in C^k_b(X) \text{ cocycle representing } \varphi \}.$$
{For more details, we refer the reader to \cite[Section 1]{Gro82}.}

Now, we recall the $\ell^1$--homology of $X$. Let $c=\sum_{i=1}^\infty a_i \sigma_i$ be a formal infinite linear sum of singular $k$--simplices in $X$ with coefficients in $\mathbb{R}$ where $\sigma_i$'s are distinct singular $k$--simplices for $k \in \mathbb{N}$. Define the $\ell^1$--norm $\|c\|_1$ of $c$ by $\|c\|_1 =\sum_{i=1}^\infty |a_i|$. Consider the set $C^{\ell^1}_k(X)$ defined by
$$C^{\ell^1}_k(X) = \left\{ c=\sum_{i=1}^\infty a_i \sigma_i \ \Bigg| \
\|c\|_1 <\infty \right\}.$$
The usual boundary operator is well defined on $C_*^{\ell^1}(X)$. Then, the \emph{$\ell^1$--homology $H^{\ell^1}_*(X)$ of $X$} is defined as the homology of the singular $\ell^1$--chain complex $C^{\ell^1}_*(X)$ of $X$. Note that the $\ell^1$--norm gives rise to a seminorm $\|\cdot \|_1$ in $H^{\ell^1}_*(X)$ as usual.

Note that the $\ell^\infty$--seminorm $\| \cdot\|_\infty$ on bounded cohomology and the $\ell^1$--seminorm $\| \cdot \|_1$ on $\ell^1$--homology are dual to each other.

{Gromov \cite[Section 1.1]{Gro82} realized that the $\ell^1$--seminorms on singular homology can be reformulated in terms of the $\ell^\infty$--seminorms on singular cohomology. Similarly, the seminorms on $\ell^1$--homology and bounded cohomology are intertwined as follows.} For each $\alpha \in H_k^{\ell^1}(X)$,
$$\| \alpha \|_1 = \sup \left\{ \frac{1}{\|\varphi \|_\infty} \ \Bigg| \ \varphi \in H^k_b(X) \text{ and } \langle \varphi, \alpha \rangle =1 \right\}.$$
For more details about this, see \cite[Theorem 3.8]{Lo07}.

\subsection{Simplicial volume of open manifolds}\label{sec:3.4}
Let $M$ be a connected, oriented, open $n$--manifold. Recall that $M$ has the fundamental class $[M]$ in $H^\mathrm{lf}_n(M)$. Let $[M]^{\ell^1}$ be the set of all $\ell^1$--homology classes that are represented by at least one locally finite fundamental cycle with finite $\ell^1$--norm. Then, the simplicial volume of $M$ can be computed in terms of $\ell^1$--homology as follows {\cite[Section 5.3]{Lo07}} :
$$ \| M \| = \inf \{ \| \alpha \|_1 \ | \ \alpha \in [M]^{\ell^1} \subset H^{\ell^1}_n(M) \}.$$
Since the $\ell^\infty$--seminorm in the bounded cohomology of $M$ is dual to the $\ell^1$--seminorm in the $\ell^1$--homology of $M$, the simplicial volume of $M$ can be computed by
$$ \| M \| = \inf_{\alpha \in [M]^{\ell^1}} \sup \left\{ \frac{1}{\| \varphi \|_\infty} \ \bigg| \ \varphi \in  H^n_b(M)\text{ and }\langle \varphi, \alpha \rangle=1 \right\}.$$
In the same way, the geometric simplicial volume of a Riemannian manifold $M$ equipped with a Riemannian metric $g_M$ can be reformulated by
$$ {\| M,g_M \|_\mathrm{Lip} = \inf_{\alpha \in [M,g_M]^{\ell^1}_\mathrm{Lip}}} \sup \left\{ \frac{1}{\| \varphi \|_\infty} \ \bigg| \ \varphi \in  H^n_b(M)\text{ and }\langle \varphi, \alpha \rangle=1 \right\},$$
where {$[M,g_M]^{\ell^1}_\mathrm{Lip}$} is the set of all $\ell^1$--homology classes that are represented by at least one locally finite fundamental cycle with finite Lipschitz constant and finite $\ell^1$--norm.
We refer the reader to \cite[Section 5.3]{Lo07} for more details.

\subsection{Compact manifolds with amenable boundary}\label{sec:3.5}
Let $V$ be a connected, oriented, compact $n$--manifold with boundary $\partial V$. Suppose that $n \geq 2$ and each path component of $\partial V$ has amenable fundamental group. Let $M$ be the interior of $V$. Choose {an open collar neighborhood} $U$ of $\partial V$. Then, $K = V-U$ is a compact submanifold of $M$ that is a deformation retract of $V$.
Clearly, each path component of $M-K$ also has an amenable fundamental group since $M-K$ is homeomorphic to $\partial V \times (0,1)$. According to Theorem \ref{thm:2.2}, the pullback map
$$ i^* \co H^n_b(M,M-K) \rightarrow H^n_b(M)$$
induced from the canonical inclusion $i \co (M,\emptyset) \rightarrow (M,M-K)$ is an isometric isomorphism.

Consider the canonical inclusion maps
$$j_1 \co (V, \partial V) \rightarrow (V,U) \text{ and } j_2 \co (M,M-K) \rightarrow (V,U).$$

{Noting $U$ and $M-K$ are homeomorphic to $\partial V \times [0,1)$ and $\partial V \times (0,1)$ respectively,
it is clear that both $j_1$ and $j_2$ are homotopy equivalences. Hence, these maps induce isometric isomorphisms in both relative bounded cohomology and relative singular homology. Then we have
\begin{eqnarray}
\| V ,\partial V \| = \| V, U \| = \| M,M-K \|.
\end{eqnarray}
where $\|V,U\|$ and $\|M,M-K \|$ are the seminorms of the fundamental classes $[V,U]\in H_n(V,U)$ and $[M,M-K]\in H_n(M,M-K)$ respectively.}

Let $\alpha \in [M]^{\ell^1}$.
Due to the isomorphism $i^* \co H^n_b(M,M-K) \rightarrow H^n_b(M)$, we can use $H^n_b(M,M-K)$ to characterize bounded cohomology classes $\varphi \in H^n_b(M)$ with $\langle \varphi, \alpha \rangle =1$ as follows.

\begin{lemma}\label{lem:3.3}
{Let $V$ be a connected, oriented, compact manifold of dimension $n$ and let $M$ be the interior of $V$.
Let us choose a compact submanifold $K$ of $M$ as above. If the inclusion $(V,\emptyset) \rightarrow (V,\partial V)$ induces an isometric isomorphism $H^n_b(V,\partial V) \rightarrow H^n_b(V)$, then for any $\ell^1$--homology class $\alpha \in [M]^{\ell^1}$, we have} $$\{\varphi \in H^n_b(M) \ | \ \langle \varphi, \alpha \rangle =1 \} = \{i^*(\beta) \ | \ \beta\in H^n_b\left(M,M-K\right), \langle \beta, [M,M-K] \rangle =1 \}.$$
\end{lemma}

\begin{proof}
{First, note that the inclusion $i:(M,\emptyset)\rightarrow (M,M-K)$ induces an isometric isomorphism $i^* : H^n_b(M,M-K) \rightarrow H^n_b(M)$ because $(M,M-K)$ and $M$ are homotopy equivalent to $(V,\partial V)$ and $V$ respectively.}

Suppose that $\langle \varphi,\alpha \rangle =1$ for $\varphi \in H^n_b(M)$.
Let $c=\sum_{i=1}^\infty a_i \sigma_i$ be a locally finite fundamental cycle with finite $\ell^1$--norm representing $\alpha$.
Since $i^* \co H^n_b(M,M-K) \rightarrow H^n_b(M)$ is an isomorphism, there exists a
cohomology class $\beta \in H^n_b(M,M-K)$ such that $i^*(\beta)=\varphi$.

Let $z \in C^n_b(M,M-K)$ be a cocycle representing $\beta$.
It follows from the assumption $\langle \varphi,\alpha \rangle =1$ that $\langle i^*(z), c \rangle =1$. Since the cocycle $z$ has compact support $K$, we have $${ 1=\langle i^*(z), c \rangle = \langle z, c|_K \rangle }$$ where $c|_K=\sum_{\mathrm{im} \sigma_i \cap K \neq \emptyset} a_i \sigma_i$.
For a locally finite fundamental cycle $c$, it is a standard fact that
$c|_L=\sum_{\mathrm{im} \sigma_i \cap L \neq \emptyset} a_i \sigma_i$ represents the fundamental class $[M,M-L]$ in $H_n(M,M-L)$ for any compact subset $L$ of $M$.
Hence, $c|_K$ represents the fundamental class $[M,M-K]$ in $H_n(M,M-K)$.
Observe that $z$ is a cocycle in $C^n(M,M-K)$ and $c|_K$ is a cycle in $C_n(M,M-K)$. From this point of view, we have $$ {1=\langle z,c|_K \rangle = \langle \beta, [M,M-K] \rangle},$$
where $\langle \beta, [M,M-K] \rangle$ is the Kronecker product in the sense of $$\langle \cdot, \cdot \rangle \co H^n(M,M-K) \otimes H_n(M,M-K) \rightarrow \mathbb{R}.$$

Conversely, suppose that $\langle \beta, [M,M-K] \rangle =1$ for $\beta \in H^n_b(M,M-K)$. In a similar way as above, one can conclude that
$${1=\langle i^*(\beta), \alpha \rangle = \langle \beta, [M,M-K] \rangle},$$
for all $\alpha \in [M]^{\ell^1}$. Therefore, this completes the proof.
\end{proof}

If $M$ is a manifold as in the assumptions of Lemma \ref{lem:3.3}, then Lemma \ref{lem:3.3} implies that the set $\{\varphi \in H^n_b(M) \ | \ \langle \varphi, \alpha \rangle =1 \}$ for $\alpha \in [M]^{\ell^1}$ is independent of the choice of $\alpha \in [M]^{\ell^1}$. Hence, one can see that for any $\alpha \in [M]^{\ell^1}$,
{\setlength\arraycolsep{2pt}
\begin{eqnarray*}
\| M \| &=& \sup \left\{ \frac{1}{\| \varphi \|_\infty} \ \bigg| \ \varphi \in  H^n_b(M)\text{ and }\langle \varphi, \alpha \rangle=1 \right\}.
\end{eqnarray*}
Furthermore if $[M,g_M]^{\ell^1}_\mathrm{Lip}$ is not empty for a given Riemannian metric $g_M$ on $M$, it follows that $\| M \| =\|M,g_M\|_\mathrm{Lip}$ (See Section \ref{sec:3.4}). From this point of view, it is not difficult to see that for a closed Riemannian manifold, its simplicial volume is equal to its geometric simplicial volume as follows. 

Let $M$ be a closed manifold with a given Riemannian metric $g_M$. Then it is clear that $M$ satisfies all the conditions of Lemma \ref{lem:3.3}. 
Since $M$ is a smooth manifold, the smooth singular homology of $M$ is isometrically isomorphic to the singular homology of $M$ (\cite[Section D.1.2]{Lo07}). Hence there is a fundamental cycle of $M$ which is a finite sum of smooth singular simplices.
A smooth fundamental cycle of $M$ induces an $\ell^1$--homology class of $M$ with finite Lipschitz constant and moreover represents the locally finite fundamental class of $M$. This implies that $[M,g_M]^{\ell^1}_\mathrm{Lip}$ is not empty. Then, as we mentioned above, it immediately follows that $\| M \| =\|M,g_M\|_\mathrm{Lip}$.

{
\begin{theorem}\label{thm:3.4}
Let $V$ be a connected, oriented, compact $n$--manifold and let $M$ be the interior of $V$.
 If the inclusion $(V,\emptyset) \rightarrow (V,\partial V)$ induces an isometric isomorphism $H^n_b(V,\partial V) \rightarrow H^n_b(V)$ and $\| M\|$ is finite, then
$$\| V,\partial V \| = \| M\|.$$
If moreover, $g_M$ is a Riemannian metric on $M$ and $\| M, g_M\|_\mathrm{Lip}$ is finite, then
$$\| V,\partial V \| = \| M\| = \| M,g_M \|_\mathrm{Lip}.$$
\end{theorem}
}

\begin{proof}
By the assumption that $\|M\|$ is finite, $[M]^{\ell^1}$ is not the empty set.
From Lemma \ref{lem:3.3}, we have
{\setlength\arraycolsep{2pt}
\begin{eqnarray*}
\| M \| &=& \inf_{\alpha \in [M]^{\ell^1}} \sup \left\{ \frac{1}{\| \varphi \|_\infty} \ \bigg| \ \varphi \in  H^n_b(M)\text{ and }\langle \varphi, \alpha \rangle=1 \right\} \\
&=& \sup \left\{ \frac{1}{\| i^*(\beta) \| _\infty} \ \bigg| \ \beta \in H^n_b(M,M-K) \text{ and } \langle \beta, [M,M-K] \rangle=1 \right\} \\
& =& \sup \left\{ \frac{1}{\| \beta \| _\infty} \ \bigg| \ \beta \in H^n_b(M,M-K) \text{ and } \langle \beta, [M,M-K] \rangle=1 \right\}\\
&=& \| M,M-K \| \\
&=& \| V,\partial V \|.
\end{eqnarray*}}
The third equation is due to the fact that $i^* \co H_b^n\left(M,M-K\right)\rightarrow H_b^n\left(M\right)$ is an isometry.
{Note that the second equation follows from Lemma \ref{lem:3.3} only if $[M]^{\ell^1}$ is not empty and it does not hold if $[M]^{\ell^1}$ is empty. This is why we need to assume that $\|M\|$ is finite.} 

As mentioned in Section \ref{sec:3.4}, for a given Riemannian metric $g_M$ on $M$, its geometric simplicial volume $\|M,g_M \|_\mathrm{Lip}$ can be computed by
$$ \| M,g_M \|_\mathrm{Lip} = \inf_{\alpha \in [M,g_M]^{\ell^1}_\mathrm{Lip}} \sup \left\{ \frac{1}{\| \varphi \|_\infty} \ \bigg| \ \varphi \in  H^n_b(M)\text{ and }\langle \varphi, \alpha \rangle=1 \right\}.$$
{
If $\| M,g_M\|_\mathrm{Lip}$ is finite, the set $[M,g_M]^{\ell^1}_\mathrm{Lip}$ is not the empty set. Then by choosing an element in $[M,g_M]^{\ell^1}_\mathrm{Lip}$, it follows from Lemma \ref{lem:3.3} that $\| M,g_M\|_\mathrm{Lip}=\|M\|$, which proves the second statement in this theorem.}
\end{proof}

{In the case that all boundary components of $V$ have amenable fundamental group and $\dim V \geq 2$, the inclusion $(V,\emptyset) \rightarrow (V,\partial V)$ induces an isometric isomorphism $H^n_b(V,\partial V) \rightarrow H^n_b(V)$ due to Theorem \ref{thm:2.2}. Furthermore, it follows from L\"{o}h's finiteness criterion for simplicial volume \cite[Theorem 6.1]{Lo07} that $\|M\|$ is finite. Hence, we obtain the following corollary from Theorem \ref{thm:3.4}.
\begin{corollary}\label{cor:3.5}
Let $V$ be a connected, oriented, compact manifold of dimension at least $2$ and let $M$ be the interior of $V$. If each path component of the boundary of $V$ has amenable fundamental group, then
$$\| V,\partial V \| = \| M\|.$$
If moreover, $g_M$ is a Riemannian metric on $M$ and $\| M, g_M\|_\mathrm{Lip}$ is finite, then
$$\| V,\partial V \| = \| M\| = \| M,g_M \|_\mathrm{Lip}.$$
\end{corollary}
}

\begin{remark}
Theorem \ref{thm:3.4} fails in the case of one-dimensional manifolds. Here is a counterexample. Consider the one-dimensional compact manifold $V=[0,1]$. Then, the relative simplicial volume $\| V, \partial V \|=1$. However,
the interior $M$ of $V$ is
the open interval $(0,1)$ and hence, $\| M \| = \infty$. Thus, Theorem \ref{thm:3.4} does not hold for one-dimensional manifolds. In fact, this is because Theorem \ref{thm:2.2} does not work in degree $1$. {Also, the assumption that $\|M,g_M\|_\mathrm{Lip}$ is finite plays an essential role in Theorem \ref{thm:3.4}. For example, even though the Poincar\'{e} hyperbolic disk $(\mathbb{H}^2,h)$ with the hyperbolic metric $h$ satisfies all assumptions except that $\|\mathbb{H}^2,h\|_\mathrm{Lip}$ is finite, the second statement of the theorem fails. More precisely, $\|\mathbb{H}^2,h\|_\mathrm{Lip}=\infty$ since $\mathbb{H}^2$ has infinite volume. However $\|\mathbb{H}^2\|=\|\mathbb{H}^2, \partial \mathbb{H}^2 \|=0. $}
\end{remark}

Let {$(M,g_M)$} be a connected, oriented, complete, pinched negatively curved manifold of finite volume.
Then, it is well known that $M$ is tame. Thus, we can suppose that $M$ is homeomorphic to the interior of a compact manifold $V$. Moreover, each path component of $\partial V$ has a virtually nilpotent fundamental group which injects in $\pi_1 V$.

\begin{corollary}\label{cor:3.6}
Let {$(M,g_M)$} be a connected, oriented, complete, pinched negatively curved manifold of finite volume that is homeomorphic to the interior of a compact manifold $V$. Then, $$\| V, \partial V \| = \| M \| = \| M,g_M\|_\mathrm{Lip}.$$
\end{corollary}

\begin{proof}
{Due to Corollary \ref{cor:3.6}, it suffices to show that $\|M,g_M\|_\mathrm{Lip}$ is finite. In fact, this follows from an estimate of Gromov for the geometric simiplicial volume: For every complete $n$--dimensional Riemannian manifold $(N,g_N)$ with sectional curvature $Sec(N,g_N)\leq 1$ and Ricci curvature $Ricci(N,g_N)\geq -(n-1)$, there is a constant $C_n>0$ such that
$$ \|N,g_N\|_\mathrm{Lip} \leq C_n \cdot \mathrm{Vol}(N,g_N).$$ This implies that $\|M,g_M\|_\mathrm{Lip}$ is finite.}
\end{proof}

Recall that the proportionality principle of simplicial volume for open manifolds fails in general. It is known, however, that the proportionality principle for $\mathbb{R}$--rank $1$ locally symmetric spaces holds (\cite{Ku12}). Now, we establish the proportionality principle for connected, complete, pinched negatively curved manifolds of finite volume by Theorems \ref{thm:3.2} and \ref{thm:3.4}, which generalizes the proportionality principle for $\mathbb{R}$--rank $1$ locally symmetric spaces.

\begin{corollary}\label{cor:3.7}
Let {$(M,g_M)$ and $(N,g_N)$} be connected, complete, pinched negatively curved manifolds of finite volume whose universal covers are isometric. Then
$$\frac{\| M \|}{\mathrm{Vol}(M,g_M)} = \frac{\| N \|}{\mathrm{Vol}(N,g_N)}.$$
\end{corollary}

\end{document}